\documentclass[a4paper, 10pt]{article}
\usepackage{amsmath, amsthm}
\usepackage{amssymb}
\IfFileExists{newtxtext.sty}{\usepackage{newtxtext}}{}
\IfFileExists{newtxmath.sty}{\usepackage{newtxmath}}{}
\usepackage{microtype}
\usepackage{bm}
\usepackage{graphicx}
\usepackage{hyperref}
\usepackage{geometry}
\usepackage{tikz}
\usepackage{epstopdf}
\usepackage{xcolor}
\IfFileExists{titlesec.sty}{\usepackage{titlesec}}{}
\IfFileExists{enumitem.sty}{\usepackage{enumitem}}{}
\IfFileExists{fancyhdr.sty}{\usepackage{fancyhdr}}{}
\usepackage{float}
\usepackage{subcaption}

\definecolor{accent}{HTML}{163A5F}
\definecolor{accentsoft}{HTML}{6F86A0}

\hypersetup{
    colorlinks=true,
    linkcolor=accent,
    citecolor=accent,
    urlcolor=accentsoft,
    pdftitle={Determinant Factorization for Left Multiplication in the Sedenions},
    pdfauthor={Shoot Koebisu}
}

\theoremstyle{plain}
\newtheorem{theorem}{Theorem}[section]

\newtheorem{lemma}[theorem]{Lemma}
\newtheorem{corollary}[theorem]{Corollary}
\theoremstyle{definition}
\newtheorem{definition}[theorem]{Definition}
\theoremstyle{remark}
\newtheorem{remark}[theorem]{Remark}

\IfFileExists{macros.tex}{\makeatletter
\renewenvironment{abstract}{
  \begin{center}
    {\large\bfseries Abstract}
  \end{center}
  \begin{center}
  \begin{minipage}{0.8\textwidth}
  \normalsize
}{
  \end{minipage}
  \end{center}
}

\renewcommand{\maketitle}{
  \begin{center}
    \hrule height 5pt
    \vspace{8mm}
    {\bfseries\LARGE \@title \par}
    \vspace{8mm}
    \hrule height 1pt
    \vspace{12mm}
    {\large \@author \par}
    \vspace{4mm}
    {\@date \par}
  \end{center}
  \vspace{4mm}
}
\makeatother}{}

\makeatletter
\renewenvironment{abstract}{
    \begin{center}
        {\color{accent}\bfseries\normalsize Abstract}
    \end{center}
    \vspace{-0.25em}
    \noindent
    \begin{center}
        \begin{minipage}{0.84\textwidth}
            \small
}{
        \end{minipage}
    \end{center}
    \vspace{0.35em}
    \normalsize
}

\renewcommand{\maketitle}{
    \thispagestyle{empty}
    \begin{center}
        {\color{accent}\rule{0.9\textwidth}{1.1pt}\par}
        \vspace{1.2em}
        {\bfseries\fontsize{17}{20}\selectfont \@title \par}
        \vspace{1.1em}
        {\normalsize \@author \par}
        \vspace{0.9em}
        {\color{accentsoft}\rule{0.5\textwidth}{0.55pt}\par}
    \end{center}
    \vspace{1.4em}
}
\makeatother

\makeatletter
\@ifpackageloaded{titlesec}{
    \titleformat{\section}
        {\large\bfseries\color{accent}}
        {\thesection}
        {0.65em}
        {}
    \titleformat{\subsection}
        {\normalsize\bfseries\color{accent}}
        {\thesubsection}
        {0.6em}
        {}
    \titleformat{name=\subsection,numberless}
        {\normalsize\bfseries\color{accent}}
        {}
        {0pt}
        {}
    \titlespacing*{\section}{0pt}{1.8ex plus .3ex minus .2ex}{0.8ex}
    \titlespacing*{\subsection}{0pt}{1.35ex plus .2ex minus .15ex}{0.55ex}
    \titlespacing*{name=\subsection,numberless}{0pt}{1.35ex plus .2ex minus .15ex}{0.5ex}
}{}

\@ifpackageloaded{enumitem}{
    \setlist[itemize]{topsep=0.4em, itemsep=0.35em, leftmargin=2.2em}
}{}

\@ifpackageloaded{fancyhdr}{
    \pagestyle{plain}
    \fancyhf{}
}
\makeatother

\title{Determinant Factorization for Left Multiplication in the Sedenions}
\author{
    {\bfseries Shoot Koebisu} \\
    \vspace{0.25em}
    {\small Hiroshima University High School, Fukuyama, Japan} \\
    {\small \href{mailto:shoot.koe@gmail.com}{shoot.koe@gmail.com}}
}

\date{}

\begin{document}
\maketitle

\begin{abstract}
    We study zero-divisors in the 16--dimensional sedenion algebra through the
    determinant of left multiplication. Write
    \[
        v=v_{1}+v_{2}e_{8},\qquad
        v_{1}=x_{1}e_{0}+u,\qquad
        v_{2}=x_{2}e_{0}+w,\qquad
        u,w\in\operatorname{Im}(\mathbb{O}),
    \]
    Then a $G_{2}$--invariant reduction to a quaternionic normal form, followed by an
    explicit block computation, gives
    \[
        \det L_{v} = D_{1}(v)^{4} D_{2}(v)^{2},
    \]
    where
    \[
        D_{1}(v)=\|v\|^{2},\qquad
        D_{2}(v)=D_{1}(v)^{2}-4\bigl(\|u\|^{2}\|w\|^{2}-\langle u,w\rangle^{2}\bigr).
    \]
    The quartic equation $D_{2}=0$ recovers the classical left zero-divisor
    condition, and after normalization the resulting manifold is
    $V_{2}(\mathbb{R}^{7})$. On a 3--dimensional purely imaginary cyclic slice we obtain
    $D_{2}(X,Y,Z)=4(XY+YZ+ZX)^{2}$, so the corresponding slice of the
    zero-divisor set is the quadratic cone $XY+YZ+ZX=0$.
\end{abstract}

\section{Introduction}
\label{sec:intro}

The sedenions arise from one further Cayley--Dickson doubling of the octonions.
Unlike the octonions, they are neither alternative nor division, and zero
divisors appear abundantly. The algebraic zero-divisor locus in
$S\cong\mathbb{R}^{16}$ is singular, so an explicit defining formula is useful.
The normalized nonzero locus considered later is smooth.

Write
\[
    v=v_{1}+v_{2}e_{8},\qquad
    v_{1}=x_{1}e_{0}+u,\qquad
    v_{2}=x_{2}e_{0}+w,\qquad
    u,w\in\operatorname{Im}(\mathbb{O}).
\]
For sedenions, the classical condition for nonzero left zero-divisors is
\[
    x_{1}=x_{2}=0,\qquad \|u\|=\|w\|,\qquad \langle u,w\rangle=0,
\]
as in Moreno~\cite{Moreno1998}. Geometric developments were later studied by
Biss--Dugger--Isaksen~\cite{BissDuggerIsaksen2008} and
Reggiani~\cite{Reggiani2024}. Those works already give geometric information
about the zero-divisor locus. What is less explicit is a determinant-based
description that starts from the intrinsic linear operator
\[
    L_v\colon S\to S
\]
and recovers the same quartic condition directly from its singularity. The aim
of this paper is to make that determinant-level description explicit.

\subsection*{Related work}
Conway--Smith~\cite{ConwaySmith2003} and Baez~\cite{Baez2002} give general
background on composition algebras and on what fails beyond the octonions.
Moreno~\cite{Moreno1998} describes sedenion zero-divisors using orthogonal
imaginary octonion pairs of equal norm. Algebraic aspects of Cayley--Dickson algebras
needed for our matrix setup appear in Benkart--Osborn~\cite{BenkartOsborn1977},
Rassias~\cite{Rassias1991}, Albuquerque--Majid~\cite{AlbuquerqueMajid1999}, and
Okubo~\cite{Okubo1978,Okubo1992}; see also Schafer~\cite{Schafer1995}.
The zero-divisor geometry in higher Cayley--Dickson algebras was developed
further by Moreno~\cite{Moreno2005Construct} and
Biss--Dugger--Isaksen~\cite{BissDuggerIsaksen2008},
who study annihilators and top-dimensional zero-divisors in a broader framework
(including Stiefel-type models of $V_{2}(\mathbb{R}^{7})$ in
double-pure/top-dimensional settings). More recently,
Reggiani~\cite{Reggiani2024} treated the sedenions themselves and introduced
$G_{2}$-invariant metrics on the corresponding normalized zero-divisor spaces.

For a sedenion $v$, being a left zero-divisor is equivalent to the singularity
of the linear map
\[
    L_v\colon S\to S,
\]
so the determinant $\det L_v$ is the most direct polynomial from which to seek
an equation for the locus. What emerges is the factorization
\[
    \Delta(v)=D_{1}(v)^{4}D_{2}(v)^{2},
\]
where the quartic factor $D_{2}$ packages, in one polynomial expression, both
the vanishing of the scalar parts and the orthogonality/equal-norm condition on
the imaginary parts. This same factor also governs the local slice considered
later. Our main theorem establishes this factorization for the $16\times16$
matrix of $L_v$, with the consequence that the quartic factor $D_2$ alone
detects the nonzero left zero-divisor locus. From the same factor we recover
the classical condition of Moreno, identify the normalized zero-divisor
manifold with $V_{2}(\mathbb{R}^{7})$, and obtain on a concrete purely
imaginary cyclic slice the quadratic-cone model studied later in the paper.

The proof proceeds by first reducing $(u,w)$, via the $G_2$ action, to a
canonical form. In those coordinates the determinant becomes a block
computation controlled by a complex $4\times4$ matrix, and the resulting
quartic factor can then be read geometrically through the normalized
zero-divisor manifold and the cyclic slice.

\section{Preliminaries}
\label{sec:prelim}

\subsection{Cayley--Dickson construction and the sedenions}

The Cayley--Dickson construction gives a procedure to obtain the complex
numbers, quaternions, octonions, and higher nonassociative algebras by successively
doubling the dimension starting from the real numbers $\mathbb{R}$. In this
paper, we take the octonion algebra $\mathbb{O}$ as known and construct the sedenion
algebra $S$ from it. We work over $\mathbb{R}$ throughout; this is the setting
relevant to the geometric sections of the paper.

As a vector space, $S$ is isomorphic to $\mathbb{O}\oplus\mathbb{O}$, and any element
$v\in S$ can be uniquely written as
\[
    v = v_{1} + v_{2} e_{8}, \qquad v_{1},v_{2}\in\mathbb{O}.
\]
Here $e_{8}$ is a new imaginary basis element. In the Cayley--Dickson model it
satisfies
\[
    e_{8}x=\overline{x}\,e_{8}\qquad (x\in\mathbb{O}),
\]
so it anticommutes with purely imaginary octonions and commutes with real
scalars. Thus, as a vector space,
\[
    S \cong \mathbb{O}\oplus\mathbb{O}\cong \mathbb{R}^{8}\times\mathbb{R}^{8} \cong
    \mathbb{R}^{16}.
\]
We use the Cayley--Dickson formula
\[
    (a+b e_{8})(c+d e_{8}) = (ac-\overline d\,b) + (da+b\overline c)e_{8},
\]
for $a,b,c,d\in\mathbb{O}$, where $\overline{\cdot}$ is octonion conjugation and
$e_{8}^{2}=-1$.
We write
\[
    \tau(x):=\overline{x}
\]
for octonion conjugation when it appears as a linear map. Thus, inside algebra
products we keep the bar notation $\overline{x}$, whereas in block-operator
expressions we write $\tau$.

For $x\in\mathbb{O}$ or $x\in S$, we use $\|x\|$ for the Euclidean norm and
write the algebraic norm as $N(x):=\|x\|^{2}$.

\subsection{Left multiplication operators and the determinant}

For $v\in S$ we define the left multiplication operator $L_{v}\colon S\to S$ by
\[
    L_{v}(w) = v w \qquad (w\in S).
\]
Once a basis $\{e_{0},\dots,e_{15}\}$ is fixed, $L_{v}$ is represented by a
$16\times 16$ real matrix $M(v)$ on $\mathbb{R}^{16}$.

\begin{definition}[Determinant of the left multiplication matrix]
    We define
    \[
        \Delta(v) := \det M(v).
    \]
\end{definition}

The determinant $\Delta(v)=\det M(v)$ is basis-independent, since it is the
determinant of the linear operator $L_v\colon S\to S$. Only its explicit
coordinate expression as a polynomial in the components of $v$ depends on the
chosen basis.

The condition $\Delta(v)=0$ means that $L_{v}$ is degenerate, and hence there
exists $w\neq 0$ such that $vw=0$. Therefore, $\Delta(v)=0$ is equivalent to
$v$ being a left zero-divisor.

\section{Determinant Factorization}
\label{sec:factor}

We compute the determinant of the left multiplication matrix
\[
    \Delta(v) := \det M(v),
\]
and factor it as
\[
    \Delta(v) = D_{1}(v)^{4} D_{2}(v)^{2}.
\]
The quartic factor $D_{2}$ depends on the scalar parts of $v_{1},v_{2}$ and on
the Gram determinant of their imaginary parts.

\subsection{Definition of the fundamental factors D\texorpdfstring{$_{1}$}{1} and
D\texorpdfstring{$_{2}$}{2}}

\begin{definition}[Fundamental factors $D_{1},D_{2}$]
    \label{def:D1D2} For a sedenion $v = v_{1} + v_{2} e_{8}$, write
    \[
        v_{1}=x_{1}e_{0}+u,\qquad v_{2}=x_{2}e_{0}+w,\qquad
        u,w\in\operatorname{Im}(\mathbb{O}).
    \]
    Define
    \[
        D_{1}(v) := \|v_{1}\|^{2} + \|v_{2}\|^{2},
    \]
    \[
        D_{2}(v) := D_{1}(v)^{2} - 4\bigl(\|u\|^{2}\|w\|^{2}-\langle u,w\rangle^{2}\bigr),
    \]
    where $\|\cdot\|$ is the Euclidean norm on $\mathbb{R}^{8}$ and
    $\langle\cdot,\cdot\rangle$ denotes the standard inner product.
\end{definition}

$D_{1}$ is simply the squared norm of $v$.

\begin{lemma}
    \label{lem:D1norm} For $v=v_{1}+v_{2}e_{8}$ we have
    \[
        \|v\|^{2} = D_{1}(v) = \|v_{1}\|^{2} + \|v_{2}\|^{2}.
    \]
\end{lemma}

\begin{proof}
    In the Cayley--Dickson construction, the inner product is defined by
    \[
        \langle v_{1}+v_{2}e_{8},\ w_{1}+w_{2}e_{8}\rangle = \langle v_{1},w_{1}\rangle
        + \langle v_{2},w_{2}\rangle.
    \]
    In particular, setting $w=v$ gives
    \[
        \|v\|^{2} = \langle v,v\rangle = \langle v_{1},v_{1}\rangle + \langle v_{2}
        ,v_{2}\rangle = \|v_{1}\|^{2} + \|v_{2}\|^{2}.
    \]
\end{proof}

The vanishing of $D_{2}$ is exactly the classical left zero-divisor condition.

\begin{lemma}
    \label{lem:D2zero} For $v=v_{1}+v_{2}e_{8}\neq 0$, we have
    \[
        D_{2}(v)=0 \quad\Longleftrightarrow\quad x_{1}=x_{2}=0,\ \|u\|=\|w\|,\ \langle u,w\rangle=0.
    \]
\end{lemma}

\begin{proof}
    Set
    \[
        A:=\|u\|^{2},\qquad B:=\|w\|^{2},\qquad \gamma:=\langle u,w\rangle.
    \]
    Then
    \[
        D_{2}(v)=\bigl(x_{1}^{2}+x_{2}^{2}+A+B\bigr)^{2}-4(AB-\gamma^{2}).
    \]
    By Cauchy--Schwarz, $\gamma^{2}\le AB$, and by the arithmetic--geometric mean
    inequality, $4AB\le (A+B)^{2}$. Hence
    \[
        D_{2}(v)\ge (A+B)^{2}-4AB+4\gamma^{2}=(A-B)^{2}+4\gamma^{2}\ge 0.
    \]
    If $D_{2}(v)=0$, then equality must hold in both inequalities above, and
    also
    \[
        \bigl(x_{1}^{2}+x_{2}^{2}+A+B\bigr)^{2}=(A+B)^{2},
    \]
    which forces $x_{1}=x_{2}=0$. Therefore
    \[
        x_{1}=x_{2}=0,\qquad A=B,\qquad \gamma=0,
    \]
    i.e.
    \[
        x_{1}=x_{2}=0,\qquad \|u\|=\|w\|,\qquad \langle u,w\rangle=0.
    \]
    Conversely, these three conditions give
    \[
        D_{2}(v)=(A+A)^{2}-4(A^{2}-0)=0.
    \]
\end{proof}

So $D_{1}$ is the norm term, while $D_{2}$ records the scalar vanishing and the
equal-norm orthogonality condition on the imaginary parts.

\subsection{Block representation of the left multiplication matrix}
\label{subsec:block}

For $v\in S$, the left multiplication operator $L_{v}\colon S\to S$ is defined
by
\[
    L_{v}(w) = v w \qquad (w\in S).
\]
If we fix a basis of the 16--dimensional real vector space $S\cong\mathbb{R}^{16}$,
then $L_{v}$ is represented by a $16\times16$ real matrix $M(v)$.

\begin{lemma}[Block matrix representation]
    \label{lem:block} For $v=v_{1}+v_{2}e_{8}\in S$, with a suitable choice of
    basis, the left multiplication matrix $M(v)$ can be written as
    \[
        M(v)=
        \begin{pmatrix}
            L_{v_1} & -R_{v_2}\circ \tau \\
            L_{v_2}\circ \tau & R_{v_1}
        \end{pmatrix},
    \]
    where $L_{v_1},L_{v_2},R_{v_1},R_{v_2}$ are the left and right multiplication
    maps on the octonions $\mathbb{O}$, and $\tau(x):=\overline{x}$ is octonion
    conjugation.
\end{lemma}

\begin{proof}
    Using the Cayley--Dickson multiplication formula
    \[
        (v_{1}+v_{2} e_{8})(w_{1}+w_{2} e_{8}) = (v_{1} w_{1} - \overline{w_2}v_{2}
        ) + (v_{2} \overline{w_1}+ w_{2} v_{1})e_{8},
    \]
    where the conjugation falls on $w_{1}$ and $w_{2}$, not on $v_{1}$ or
    $v_{2}$. For $w=w_{1}+w_{2}e_{8}$, in
    $(w_{1},w_{2})\in\mathbb{O}\oplus\mathbb{O}$ coordinates this is
    \[
        (w_{1},w_{2}) \longmapsto \bigl(v_{1}w_{1} - (\tau w_{2})v_{2},\ \ v_{2}
        (\tau w_{1}) + w_{2}v_{1}\bigr),
    \]
    so the first component is
    \[
        v_{1}w_{1}-(\tau w_{2})v_{2}
        =L_{v_{1}}w_{1}-R_{v_{2}}(\tau w_{2}),
    \]
    and the second component is
    \[
        v_{2}(\tau w_{1})+w_{2}v_{1}
        =(L_{v_{2}}\circ\tau)w_{1}+R_{v_{1}}w_{2}.
    \]
    Taking matrix representations with respect to a fixed basis yields the claimed
    block matrix.
\end{proof}

Since the octonions form a normed division algebra, the determinants of the left
and right multiplication matrices are determined solely by the norm.

\begin{lemma}[Determinants of left and right multiplication on the octonions]
    \label{lem:octdet} For any $x\in\mathbb{O}$ we have
    \[
        \det L_{x} = \|x\|^{8},\qquad \det R_{x} = \|x\|^{8}.
    \]
\end{lemma}

\begin{proof}
    If $x=0$, then both $L_{x}$ and $R_{x}$ are the zero map, so the stated
    determinant formula is immediate. Assume now that $x\neq 0$.

    Since the octonions form a normed division algebra, for any $x,y\in\mathbb{O}$
    we have
    \[
        \|x y\| = \|x\|\,\|y\|
    \]
    and therefore
    \[
        \|L_{x}(y)\| = \|xy\| = \|x\|\,\|y\|.
    \]
    Thus $\|Q_x y\|=\|y\|$ for all $y\in\mathbb{O}$. Since $Q_x$ is linear, the
    polarization identity gives, for all $y,z\in\mathbb{O}$,
    \[
        2\langle Q_x y,Q_x z\rangle
        =\|Q_x y+Q_x z\|^2-\|Q_x y\|^2-\|Q_x z\|^2
        =\|y+z\|^2-\|y\|^2-\|z\|^2
        =2\langle y,z\rangle.
    \]
    Hence $Q_x$ preserves the Euclidean inner product.
    Thus
    \[
        Q_{x}:=\|x\|^{-1}L_{x}
    \]
    is an isometry of the Euclidean space $\mathbb{O}\cong\mathbb{R}^{8}$, so
    $Q_{x}\in O(8)$. The map $x\mapsto Q_{x}$ is continuous on
    $\mathbb{O}\setminus\{0\}$, and $\mathbb{O}\setminus\{0\}\cong
    \mathbb{R}^{8}\setminus\{0\}$ is connected. Since $Q_{1}=I$, we have
    $\det(Q_{x})=1$ for all $x\neq 0$. Consequently
    \[
        \det L_{x} = \|x\|^{8}\det(Q_{x})=\|x\|^{8}.
    \]
    The same argument, applied to $R_{x}$, gives $\det R_{x}=\|x\|^{8}$.
\end{proof}

\begin{lemma}[Quaternionic normal form for the imaginary parts]
    \label{lem:g2normal} For any $u,w\in\operatorname{Im}(\mathbb{O})$, there exist
    $g\in G_{2}$ and real numbers $r,p,q$ such that
    \[
        gu = r e_{1},\qquad gw = p e_{1}+q e_{2}.
    \]
\end{lemma}

\begin{proof}
    If $u=0$, then if also $w=0$ the claim is immediate. If $u=0$ and $w\neq 0$,
    the transitivity of the $G_2$-action on the unit sphere in
    $\operatorname{Im}(\mathbb{O})$ gives $g\in G_2$ with
    $gw=\|w\|e_1$, so the conclusion holds with $r=0$, $p=\|w\|$, and $q=0$.
    Assume now that $u\neq 0$. Since $G_{2}$ acts
    transitively on the unit sphere in $\operatorname{Im}(\mathbb{O})$, there
    exists $g_{1}\in G_{2}$ with $g_{1}u=\|u\|e_{1}$. The stabilizer of $e_{1}$
    in $G_{2}$ is isomorphic to $SU(3)$, and its action on
    $e_{1}^{\perp}\cong\mathbb{R}^{6}$ is transitive on the unit sphere
    (see Baez~\cite{Baez2002} or Bryant~\cite{Bryant2005}). Therefore, after
    composing with a suitable element $h$ of this stabilizer, we may arrange that
    the component of $g_{1}w$ orthogonal to $e_{1}$ lies on the $e_{2}$-axis.
    Setting $g:=hg_{1}$, we then have
    \[
        gu=\|u\|e_{1},\qquad gw=p e_{1}+q e_{2}
    \]
    for suitable real numbers $p,q$.
\end{proof}

Figure~\ref{fig:g2-reduction} shows this reduction. Any pair $(u,w)$ can be
moved by $G_{2}$ into the plane spanned by $e_{1}$ and $e_{2}$, where the
determinant calculation is explicit.

\begin{figure}[t]
    \centering
    \includegraphics[width=0.88\textwidth]{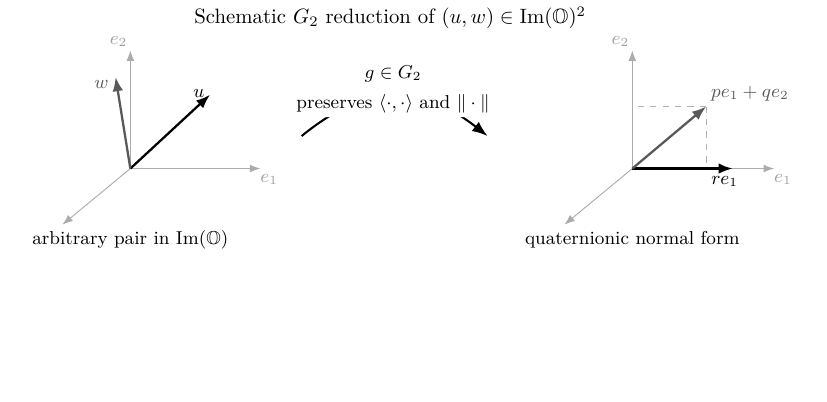}
    \caption{Schematic view of the $G_{2}$ reduction used in
    Lemma~\ref{lem:g2normal}. An arbitrary pair
    $(u,w)\in \operatorname{Im}(\mathbb{O})^{2}$ is carried to the normal form
    $(r e_{1},\, p e_{1}+q e_{2})$, preserving the norm and inner product.}
    \label{fig:g2-reduction}
\end{figure}

\subsection{Main theorem on determinant factorization}

\begin{theorem}[Determinant factorization for sedenion left multiplication]
    \label{thm:factor} For any $v\in S$, the determinant
    \[
        \Delta(v) := \det M(v)
    \]
    of the left multiplication matrix factors as
    \[
        \Delta(v) = D_{1}(v)^{4}\,D_{2}(v)^{2}.
    \]
    In particular, $v\neq 0$ is a left zero-divisor if and only if $D_{2}(v)=0$.
\end{theorem}

\begin{proof}
    Write
    \[
        v_{1}=x_{1}e_{0}+u,\qquad v_{2}=x_{2}e_{0}+w,
        \qquad u,w\in\operatorname{Im}(\mathbb{O}).
    \]

    The automorphism group $G_{2}$ of the octonions preserves the norm, the
    inner product, octonion conjugation, and the Cayley--Dickson multiplication.
    Extending the action to $S\cong\mathbb{O}\oplus\mathbb{O}$ by
    \[
        g(v_{1}+v_{2}e_{8}) := gv_{1}+gv_{2}e_{8},
    \]
    we obtain the orthogonal transformation
    \[
        T_g := g\oplus g \in O(S)
    \]
    satisfying
    \[
        M(gv)=T_{g}M(v)T_{g}^{-1}
    \]
    for all $g\in G_2$. Therefore
    \[
        \Delta(gv)=\Delta(v).
    \]
    Since $D_{1}$ and $D_{2}$ are also $G_{2}$--invariant, Lemma~\ref{lem:g2normal}
    allows us to assume
    \[
        u=r e_{1},\qquad w=p e_{1}+q e_{2}
    \]
    for some $r,p,q\in\mathbb{R}$. Set
    \[
        a:=x_{1}e_{0}+r e_{1},\qquad b:=x_{2}e_{0}+p e_{1}+q e_{2}.
    \]
    Then $a,b$ lie in the quaternionic subalgebra
    \[
        \mathbb{H}:=\operatorname{span}\{e_{0},e_{1},e_{2},e_{3}\}\subset\mathbb{O}.
    \]

    Write $\mathbb{O}=\mathbb{H}\oplus \mathbb{H}e_{4}$ and decompose
    \[
        w_{1}=x+y e_{4},\qquad w_{2}=z+t e_{4},
        \qquad x,y,z,t\in\mathbb{H}.
    \]
    Using the Cayley--Dickson formula inside $\mathbb{O}$,
    \[
        (x+y e_{4})(z+t e_{4})=(xz-\overline{t}\,y)+(tx+y\overline{z})e_{4},
    \]
    a direct computation gives
    \begin{align*}
        (a+b e_{8})\bigl((x+y e_{4})+(z+t e_{4})e_{8}\bigr)
         & = \bigl(a x-\overline{z}\,b + (y a+t\overline{b})e_{4}\bigr) \\
         & \quad + \bigl(b\overline{x}+z a +(-y b+t\overline{a})e_{4}\bigr)e_{8}.
    \end{align*}
    Let $P$ be the permutation matrix corresponding to the same reordering
    $(x,y,z,t)\mapsto(x,z,y,t)$ on both domain and codomain coordinates.
    Then the conjugated matrix $P M(v) P^{-1}$ becomes block diagonal:
    \[
        P M(v) P^{-1} =
        \begin{pmatrix}
            F_{a,b} & 0 \\
            0 & G_{a,b}
        \end{pmatrix},
    \]
    where
    \[
        F_{a,b}(x,z)=\bigl(a x-\overline{z}\,b,\ b\overline{x}+z a\bigr),
    \]
    \[
        G_{a,b}(y,t)=\bigl(y a+t\overline{b},\ -y b+t\overline{a}\bigr).
    \]
    Hence
    \[
        \Delta(v)=\det_{\mathbb{R}}F_{a,b}\cdot \det_{\mathbb{R}}G_{a,b}.
    \]

    To identify the first block, apply the Cayley--Dickson formula in
    $\mathbb{O}=\mathbb{H}\oplus\mathbb{H}e_{4}$:
    \[
        (a+b e_{4})(x+z e_{4})=(a x-\overline{z}\,b)+(z a+b\overline{x})e_{4}.
    \]
    Therefore, under the identification
    $\mathbb{H}\oplus\mathbb{H}e_{4}\cong\mathbb{H}^{2}$, we have
    \[
        L_{a+b e_{4}}(x,z)=\bigl(a x-\overline{z}\,b,\ z a+b\overline{x}\bigr)
        =F_{a,b}(x,z).
    \]
    Thus $F_{a,b}$ is exactly left multiplication by $a+b e_{4}$, and
    Lemma~\ref{lem:octdet} gives
    \[
        \det_{\mathbb{R}}F_{a,b}=\|a+b e_{4}\|^{8}
        =(\|a\|^{2}+\|b\|^{2})^{4}
        =D_{1}(v)^{4}.
    \]
    Here $\|a\|^{2}=x_{1}^{2}+r^{2}=\|v_{1}\|^{2}$ and
    $\|b\|^{2}=x_{2}^{2}+p^{2}+q^{2}=\|v_{2}\|^{2}$ in the chosen normal form.

    For the second block, write
    \[
        \alpha:=x_{1}+ir,\qquad \beta:=x_{2}+ip\in\mathbb{C}.
    \]
    Viewing $\mathbb{H}$ as a left $\mathbb{C}$-vector space via left
    multiplication by $e_{1}$, note that $\mathbb{C}=\operatorname{span}\{e_{0},e_{1}\}$
    lies in the associative algebra $\mathbb{H}$, so for every $\lambda\in\mathbb{C}$
    and $y,t\in\mathbb{H}$ we have
    \[
        G_{a,b}(\lambda y,\lambda t)=\lambda\,G_{a,b}(y,t).
    \]
    Thus $G_{a,b}$ is $\mathbb{C}$-linear. Right multiplication by
    $a=x_{1}e_{0}+r e_{1}$ acts as $\operatorname{diag}(\alpha,\overline{\alpha})$
    and right multiplication by $e_{1}$ acts as $\operatorname{diag}(i,-i)$.
    Relative to the induced complex basis of $\mathbb{H}^{2}$, the real
    $8\times 8$ matrix of $G_{a,b}$ is therefore the realification of
    the complex $4\times 4$ matrix
    \[
        \mathcal{G}(\alpha,\beta,q)=
        \begin{pmatrix}
            \alpha & 0 & \overline{\beta} & q \\
            0 & \overline{\alpha} & -q & \beta \\
            -\beta & q & \overline{\alpha} & 0 \\
            -q & -\overline{\beta} & 0 & \alpha
        \end{pmatrix}.
    \]
    Therefore
    \[
        \det_{\mathbb{R}}G_{a,b}=\bigl|\det_{\mathbb{C}}\mathcal{G}(\alpha,\beta,q)\bigr|^{2}.
    \]
    Appendix~\ref{app:canonical} gives the exact determinant formula
    \[
        \det_{\mathbb{C}}\mathcal{G}(\alpha,\beta,q)
        =\bigl(|\alpha|^{2}+|\beta|^{2}+q^{2}\bigr)^{2}-4(\operatorname{Im}\alpha)^{2}q^{2}.
    \]
    Since Appendix~\ref{app:canonical} also shows that
    $\det_{\mathbb{C}}\mathcal{G}(\alpha,\beta,q)$ is real, we may identify
    $\bigl|\det_{\mathbb{C}}\mathcal{G}(\alpha,\beta,q)\bigr|^{2}$ with
    $\bigl(\det_{\mathbb{C}}\mathcal{G}(\alpha,\beta,q)\bigr)^{2}$.
    Since
    \[
        |\alpha|^{2}=\|v_{1}\|^{2},\qquad
        |\beta|^{2}+q^{2}=\|v_{2}\|^{2},
    \]
    and
    \[
        (\operatorname{Im}\alpha)^{2}q^{2}
        =r^{2}q^{2}
        =\|u\|^{2}\|w\|^{2}-\langle u,w\rangle^{2},
    \]
    we obtain
    \[
        \det_{\mathbb{R}}G_{a,b}=D_{2}(v)^{2}.
    \]
    Hence
    \[
        \Delta(v)=D_{1}(v)^{4}D_{2}(v)^{2}.
    \]

    Finally, $v\neq 0$ is a left zero-divisor if and only if $\Delta(v)=0$, and
    Lemma~\ref{lem:D1norm} gives $D_{1}(v)=\|v\|^{2}>0$ for $v\neq 0$. Therefore
    \[
        \Delta(v)=0 \quad\Longleftrightarrow\quad D_{2}(v)=0.
    \]
\end{proof}

\begin{corollary}[Characterization of nonzero left zero-divisors]
    \label{cor:zerodiv} Write
    \[
        v=v_{1}+v_{2}e_{8},\qquad v_{1}=x_{1}e_{0}+u,\qquad v_{2}=x_{2}e_{0}+w,
        \qquad u,w\in\operatorname{Im}(\mathbb{O}).
    \]
    Then $v\neq 0$ is a left zero-divisor if and only
    if
    \[
        x_{1}=x_{2}=0,\qquad \|u\|=\|w\|,\qquad \langle u,w\rangle=0
    \]
    hold.
\end{corollary}

\begin{proof}
    This follows immediately from Theorem~\ref{thm:factor} and Lemma~\ref{lem:D2zero}.
\end{proof}

\subsection{Component expressions for D\texorpdfstring{$_{1}$}{1} and D\texorpdfstring{$_{2}$}{2}}

Write a sedenion $v\in S$ in components as
\[
    v = \sum_{m=0}^{15}a_{m} e_{m},\qquad a_{m}\in\mathbb{R}.
\]
Then $D_{1},D_{2}$ can also be expressed as polynomials in these components.

\begin{theorem}[Component expressions for the factors]
    \label{thm:component}
    \begin{align*}
        D_{1}(v) & = \sum_{m=0}^{15}a_{m}^{2}, \\
        D_{2}(v) & =
        \left(\sum_{m=0}^{15}a_{m}^{2}\right)^{2}
        -4\left[
        \left(\sum_{i=1}^{7}a_{i}^{2}\right)\left(\sum_{i=9}^{15}a_{i}^{2}\right)
        -\left(\sum_{i=1}^{7}a_{i}a_{i+8}\right)^{2}
        \right].
    \end{align*}
\end{theorem}

\begin{proof}
    For $D_{1}$, Lemma~\ref{lem:D1norm} gives
    \[
        D_{1}(v)=\|v\|^{2},
    \]
    which is clearly the sum of the squares of the components
    $\sum_{m=0}^{15}a_{m}^{2}$.

    For $D_{2}$, write
    \[
        v_{1}=a_{0}e_{0}+\sum_{i=1}^{7}a_{i}e_{i},\qquad
        v_{2}=a_{8}e_{0}+\sum_{i=1}^{7}a_{i+8}e_{i},
    \]
    so that
    \[
        x_{1}=a_{0},\qquad x_{2}=a_{8},
    \]
    \[
        \|u\|^{2}=\sum_{i=1}^{7}a_{i}^{2},\qquad
        \|w\|^{2}=\sum_{i=9}^{15}a_{i}^{2},\qquad
        \langle u,w\rangle=\sum_{i=1}^{7}a_{i}a_{i+8}.
    \]
    Substituting these into
    \[
        D_{2}(v)=\bigl(\|v_{1}\|^{2}+\|v_{2}\|^{2}\bigr)^{2}
        -4\bigl(\|u\|^{2}\|w\|^{2}-\langle u,w\rangle^{2}\bigr)
    \]
    gives exactly the stated component expression.
\end{proof}

\section{The Normalized Left Zero--Divisor Manifold}
\label{sec:stiefel}

By Corollary~\ref{cor:zerodiv}, every nonzero left zero-divisor has the form
\[
    v=u+w e_{8},
    \qquad
    u,w\in\operatorname{Im}(\mathbb{O}),
    \qquad
    \|u\|=\|w\|,\ \langle u,w\rangle=0.
\]
Thus the normalized left zero-divisor set lives in the imaginary octonions and
is governed by the Stiefel manifold in dimension $7$.

\begin{definition}[Normalized left zero-divisor manifold]
    \label{def:E} Define
    \[
        E:=\Bigl\{(u,w)\in\operatorname{Im}(\mathbb{O})\times\operatorname{Im}(\mathbb{O})
        \ \Bigm|\
        \|u\|=\|w\|=1,\ \langle u,w\rangle=0\Bigr\}.
    \]
\end{definition}

Since $\operatorname{Im}(\mathbb{O})\cong\mathbb{R}^{7}$, this is exactly the
Stiefel manifold
\[
    E\cong V_{2}(\mathbb{R}^{7}).
\]
Equivalently, the projection to the first vector,
\[
    \pi\colon V_{2}(\mathbb{R}^{7})\to S^{6},\qquad \pi(u,w)=u,
\]
has fiber
\[
    \pi^{-1}(u)=\{w\in u^{\perp}\subset\mathbb{R}^{7}\mid \|w\|=1\}\cong S^{5}.
\]
Thus $V_{2}(\mathbb{R}^{7})$ is an $S^{5}$--bundle over $S^{6}$.

\begin{theorem}[Structure of the set of left zero-divisors]
    \label{thm:global-structure} Let $Z\subset S\setminus\{0\}$ be the set of
    nontrivial left zero-divisors. Then
    \[
        Z \cong (0,\infty)\times V_{2}(\mathbb{R}^{7})
    \]
    as smooth manifolds. In particular,
    \[
        \dim Z = 1+\dim V_{2}(\mathbb{R}^{7})=1+(2\cdot 7-3)=12.
    \]
\end{theorem}

\begin{proof}
    Define
    \[
        \Phi\colon (0,\infty)\times V_{2}(\mathbb{R}^{7})\to Z,
        \qquad
        \Phi(r,(u,w)):=r(u+w e_{8}).
    \]
    Because $(u,w)$ is an orthonormal pair in $\operatorname{Im}(\mathbb{O})$,
    Corollary~\ref{cor:zerodiv} shows that $\Phi(r,(u,w))$ is a nontrivial left
    zero-divisor.

    Conversely, if $v\in Z$, then Corollary~\ref{cor:zerodiv} gives
    \[
        v=u+w e_{8},
        \qquad
        u,w\in\operatorname{Im}(\mathbb{O}),
        \qquad
        \|u\|=\|w\|=:r>0,\ \langle u,w\rangle=0.
    \]
    Hence
    \[
        v=\Phi\!\left(r,\left(\frac{u}{r},\frac{w}{r}\right)\right),
    \]
    and $\Phi$ is bijective with inverse
    \[
        \Psi(v)=\left(\|u\|,\left(\frac{u}{\|u\|},\frac{w}{\|u\|}\right)\right).
    \]
    Since $u\neq 0$ on $Z$, the norm $\|u\|$ and division by $\|u\|$ are smooth
    on $Z$. Hence $\Psi$ is smooth, and therefore $\Phi$ is a diffeomorphism.
\end{proof}

\section{A Purely Imaginary Cyclic Slice}
\label{sec:local}

A convenient place to see the quartic factor geometrically is the following
$3$--parameter slice inside the pure-imaginary subspace. On this slice
$D_{2}$ collapses to a quadratic cone.

\begin{definition}[Purely imaginary cyclic slice]
    For parameters $X,Y,Z\in\mathbb{R}$, define
    \[
        v_{1}(X,Y,Z)=(0,\,X,\,Y,\,Z,\,0,0,0,0),
    \]
    \[
        v_{2}(X,Y,Z)=(0,\,Y,\,Z,\,X,\,0,0,0,0),
    \]
    \[
        v(X,Y,Z):=v_{1}(X,Y,Z)+v_{2}(X,Y,Z)e_{8}\in S,
    \]
    and set
    \[
        L:=\{v(X,Y,Z)\mid X,Y,Z\in\mathbb{R}\}\subset S.
    \]
\end{definition}

\begin{lemma}
    \label{lem:norm-inner} On the slice $L$, we have
    \[
        \|v_{1}\|^{2}=\|v_{2}\|^{2}=X^{2}+Y^{2}+Z^{2},
    \]
    \[
        \langle v_{1},v_{2}\rangle=XY+YZ+ZX.
    \]
\end{lemma}

\begin{proof}
    This follows by direct inspection of the coordinates.
\end{proof}

\begin{theorem}[Local prototype on the cyclic slice]
    \label{thm:local} On $L$, the quartic factor is
    \[
        D_{2}(X,Y,Z)=4(XY+YZ+ZX)^{2}.
    \]
    Hence, on this slice, the left zero-divisor condition is
    \[
        XY+YZ+ZX=0.
    \]
\end{theorem}

\begin{proof}
    On the slice $L$, both scalar parts vanish, so $v_{1},v_{2}$ are purely
    imaginary. Therefore Definition~\ref{def:D1D2} reduces to
    \[
        D_{2}=(\|v_{1}\|^{2}-\|v_{2}\|^{2})^{2}+4\langle v_{1},v_{2}\rangle^{2}.
    \]
    Using Lemma~\ref{lem:norm-inner},
    \[
        \|v_{1}\|^{2}-\|v_{2}\|^{2}=0,\qquad
        \langle v_{1},v_{2}\rangle=XY+YZ+ZX,
    \]
    and hence
    \[
        D_{2}(X,Y,Z)=4(XY+YZ+ZX)^{2}.
    \]
    Corollary~\ref{cor:zerodiv} then shows that the left zero-divisor condition
    on this slice is exactly $XY+YZ+ZX=0$.
\end{proof}

Figure~\ref{fig:cyclic-slice-visuals} shows the same equation in two forms.
The three-dimensional rendering is the two-napped cone
$XY+YZ+ZX=0$, while the affine section $Z=1$ turns it into the hyperbola
$(X+1)(Y+1)=1$. This is the local picture of the vanishing set of $D_{2}$ on
the slice.

\begin{figure}[t]
    \centering
    \begin{subfigure}[t]{0.84\textwidth}
        \centering
        \includegraphics[width=\textwidth]{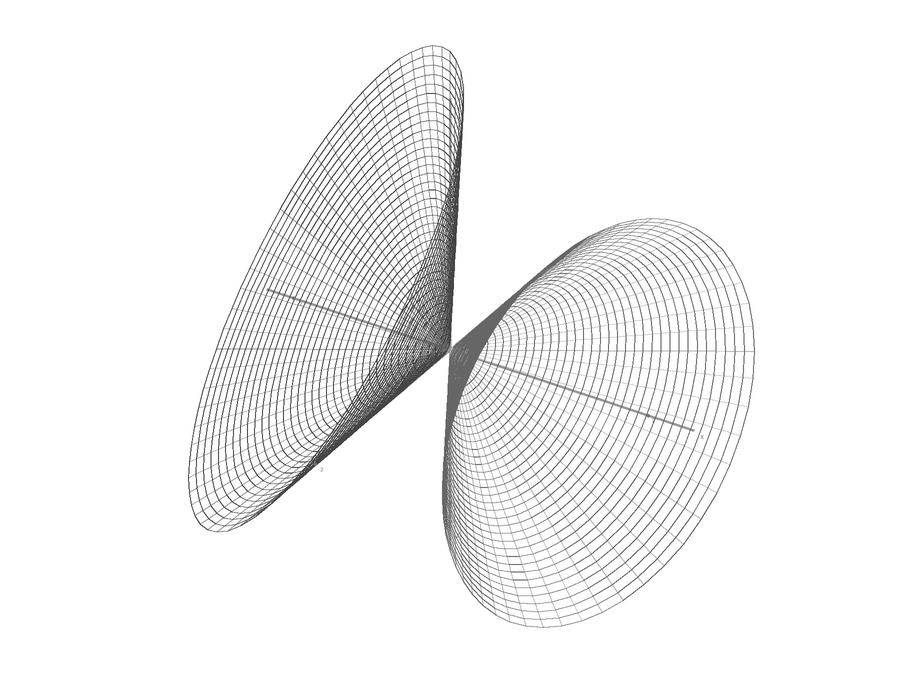}
        \caption{A three-dimensional rendering of the cyclic-slice zero-divisor cone
        $XY+YZ+ZX=0$. The two nappes reflect the real signature $(1,2)$ of the
        quadratic form.}
    \end{subfigure}

    \vspace{0.5em}

    \begin{subfigure}[t]{0.74\textwidth}
        \centering
        \includegraphics[width=\textwidth]{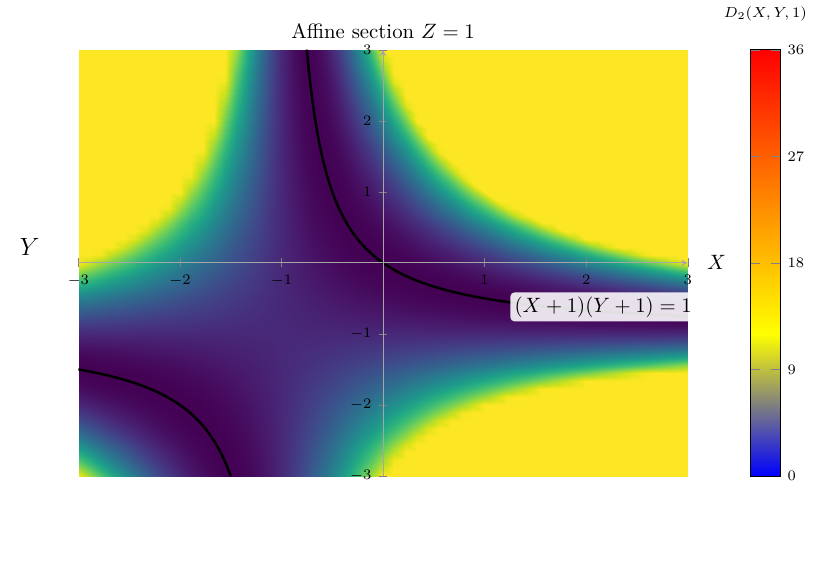}
        \caption{The affine section $Z=1$, where
        $D_{2}(X,Y,1)=4(XY+X+Y)^{2}$ and the zero set becomes the hyperbola
        $(X+1)(Y+1)=1$.}
    \end{subfigure}
    \caption{Visualizations of the quartic factor on the purely imaginary cyclic slice.}
    \label{fig:cyclic-slice-visuals}
\end{figure}

\begin{remark}[Singularity type]
    The quadratic form
    \[
        Q(X,Y,Z)=XY+YZ+ZX
    \]
    has Hessian
    \[
        H(Q)=
        \begin{pmatrix}
            0 & 1 & 1 \\
            1 & 0 & 1 \\
            1 & 1 & 0
        \end{pmatrix},
    \]
    whose eigenvalues are $2,-1,-1$. Hence $Q$ has real signature $(1,2)$.
    Thus the cone $Q=0$ has an isolated ordinary double point at the origin,
    i.e. an $A_{1}$ singularity on this slice; equivalently, its tangent cone at
    the origin is the same indefinite quadratic cone $Q=0$. In particular, on
    this three-dimensional slice the origin is the unique singular point of the
    zero-divisor cone.
\end{remark}

\section{Conclusion}
\label{sec:conclusion}

The determinant of left multiplication isolates the same
quartic condition that governs sedenion left zero-divisors. Once the
factorization of $\Delta(v)$ is in place, the classical description of the
nonzero zero-divisor locus, its normalized Stiefel model
$V_{2}(\mathbb{R}^{7})$, and the cyclic-slice cone $XY+YZ+ZX=0$ all follow
from the same polynomial factor.
In this way, the determinant viewpoint ties the global zero-divisor condition
to the local geometry seen on the cyclic slice.

\appendix

\section{The complex determinant in Theorem~\ref{thm:factor}}
\label{app:canonical}

In the notation of the proof of Theorem~\ref{thm:factor}, the block $G_{a,b}$
is the realification of
\[
    \mathcal{G}(\alpha,\beta,q)=
    \begin{pmatrix}
        \alpha & 0 & \overline{\beta} & q \\
        0 & \overline{\alpha} & -q & \beta \\
        -\beta & q & \overline{\alpha} & 0 \\
        -q & -\overline{\beta} & 0 & \alpha
    \end{pmatrix}.
\]
After simultaneously permuting rows and columns by the order $(1,3,2,4)$,
i.e. after conjugating by the corresponding permutation matrix, we obtain
\[
    \mathcal{G}'=
    \begin{pmatrix}
        \alpha & \overline{\beta} & 0 & q \\
        -\beta & \overline{\alpha} & q & 0 \\
        0 & -q & \overline{\alpha} & \beta \\
        -q & 0 & -\overline{\beta} & \alpha
    \end{pmatrix}.
\]
Thus $\det_{\mathbb{C}}\mathcal{G}'=\det_{\mathbb{C}}\mathcal{G}$. Write
\[
    \mathcal{G}'=
    \begin{pmatrix}
        A & B \\
        C & D
    \end{pmatrix},
\]
where
\[
    A=
    \begin{pmatrix}
        \alpha & \overline{\beta} \\
        -\beta & \overline{\alpha}
    \end{pmatrix},\quad
    B=
    \begin{pmatrix}
        0 & q \\
        q & 0
    \end{pmatrix},\quad
    C=
    \begin{pmatrix}
        0 & -q \\
        -q & 0
    \end{pmatrix},\quad
    D=
    \begin{pmatrix}
        \overline{\alpha} & \beta \\
        -\overline{\beta} & \alpha
    \end{pmatrix}.
\]
Set
\[
    s:=|\alpha|^{2}+|\beta|^{2}.
\]
If $s=0$, then $\alpha=\beta=0$ and a direct inspection gives
\[
    \det_{\mathbb{C}}\mathcal{G}(0,0,q)=q^{4},
\]
which already agrees with the desired formula. Assume now that $s\neq 0$.
Then
\[
    \det A=s,\qquad
    A^{-1}=\frac{1}{s}
    \begin{pmatrix}
        \overline{\alpha} & -\overline{\beta} \\
        \beta & \alpha
    \end{pmatrix},
\]
so the Schur complement formula gives
\[
    \det_{\mathbb{C}}\mathcal{G}(\alpha,\beta,q)
    =\det A\cdot\det(D-CA^{-1}B).
\]
A direct multiplication yields
\[
    CA^{-1}B
    =\frac{q^{2}}{s}
    \begin{pmatrix}
        -\alpha & -\beta \\
        \overline{\beta} & -\overline{\alpha}
    \end{pmatrix},
\]
hence
\[
    D-CA^{-1}B=
    \begin{pmatrix}
        \overline{\alpha}+\dfrac{\alpha q^{2}}{s}
        & \beta\!\left(1+\dfrac{q^{2}}{s}\right) \\
        -\overline{\beta}\!\left(1+\dfrac{q^{2}}{s}\right)
        & \alpha+\dfrac{\overline{\alpha} q^{2}}{s}
    \end{pmatrix}.
\]
Its determinant is
\begin{align*}
    \det(D-CA^{-1}B)
    &=
    \left(\overline{\alpha}+\frac{\alpha q^{2}}{s}\right)
    \left(\alpha+\frac{\overline{\alpha} q^{2}}{s}\right)
    +|\beta|^{2}\left(1+\frac{q^{2}}{s}\right)^{2} \\
    &=
    s+\frac{q^{2}}{s}\bigl(\alpha^{2}+\overline{\alpha}^{\,2}+2|\beta|^{2}\bigr)
    +\frac{q^{4}}{s}.
\end{align*}
Multiplying by $\det A=s$ gives
\[
    \det_{\mathbb{C}}\mathcal{G}(\alpha,\beta,q)
    =s^{2}+q^{2}\bigl(\alpha^{2}+\overline{\alpha}^{\,2}+2|\beta|^{2}\bigr)+q^{4}.
\]
Since $2s-2|\alpha|^{2}=2|\beta|^{2}$, this can be rewritten as
\[
    \det_{\mathbb{C}}\mathcal{G}(\alpha,\beta,q)
    =(s+q^{2})^{2}+q^{2}\bigl(\alpha^{2}+\overline{\alpha}^{\,2}-2|\alpha|^{2}\bigr).
\]
Since $\alpha^{2}+\overline{\alpha}^{\,2}-2|\alpha|^{2}
=(\alpha-\overline{\alpha})^{2}$, we obtain
\[
    \det_{\mathbb{C}}\mathcal{G}(\alpha,\beta,q)
    =(\alpha\overline{\alpha}+\beta\overline{\beta}+q^{2})^{2}
    +q^{2}(\alpha-\overline{\alpha})^{2}.
\]
Since $\alpha-\overline{\alpha}=2i\,\operatorname{Im}\alpha$, this is
\[
    \det_{\mathbb{C}}\mathcal{G}(\alpha,\beta,q)
    =\bigl(|\alpha|^{2}+|\beta|^{2}+q^{2}\bigr)^{2}
    -4(\operatorname{Im}\alpha)^{2}q^{2}.
\]
The right-hand side is real for all $\alpha,\beta,q$, so the determinant of the
realification is its square:
\[
    \det_{\mathbb{R}}G_{a,b}
    =\bigl(\det_{\mathbb{C}}\mathcal{G}(\alpha,\beta,q)\bigr)^{2}.
\]
This is the identity used in the proof of Theorem~\ref{thm:factor}.

\newpage
\bibliographystyle{plain}

\end{document}